\documentclass{amsart}

\usepackage{amssymb}
\usepackage{amsmath}
\usepackage{lscape}
\usepackage{enumerate}
\usepackage{rotating}
\usepackage{url}
\usepackage{here}

\usepackage{mathtools}

\newcommand{\erase}[1]{}

\newtheorem{theorem}{Theorem}[section]
\newtheorem{lemma}[theorem]{Lemma}
\newtheorem{proposition}[theorem]{Proposition}
\newtheorem{corollary}[theorem]{Corollary}

\newtheorem{_algorithm}[theorem]{Algorithm}

\newtheorem{_procedure}[theorem]{Procedure}

\newtheorem{_definition}[theorem]{Definition}
\newenvironment{definition}{\begin{_definition}\rm}{\end{_definition}}

\newtheorem{_remark}[theorem]{\it Remark}
\newenvironment{remark}{\begin{_remark}\rm}{\end{_remark}}

\newtheorem{_example}[theorem]{Example}

\newtheorem{_assumption}[theorem]{Assumption}

\newtheorem{_construction}[theorem]{Construction}

\newtheorem{_claim}[theorem]{Claim}

\newtheorem{_conjecture}[theorem]{Conjecture}

\numberwithin{equation}{section}
\numberwithin{table}{section}
\numberwithin{figure}{section}

\newcommand{\F}{\mathord{\mathbb F}}

\newcommand{\M}{\mathord{\mathbb M}}
\renewcommand{\P}{\mathord{\mathbb  P}}
\newcommand{\Q}{\mathord{\mathbb  Q}}
\newcommand{\R}{\mathord{\mathbb R}}

\newcommand{\Z}{\mathord{\mathbb Z}}

\newcommand{\AAA}{\mathord{\mathcal A}}

\newcommand{\CCC}{\mathord{\mathcal C}}

\newcommand{\HHH}{\mathord{\mathcal H}}

\newcommand{\LLL}{\mathord{\mathcal L}}

\newcommand{\NNN}{\mathord{\mathcal N}}

\newcommand{\PPP}{\mathord{\mathcal P}}

\newcommand{\SSS}{\mathord{\mathcal S}}
\newcommand{\TTT}{\mathord{\mathcal T}}

\newcommand{\maprightsp}[1]{\; \smash{\mathop{\; \longrightarrow \; }\limits\sp{#1}}\; }
\newcommand{\maprightsb}[1]{\; \smash{\mathop{\; \longrightarrow \; }\limits\sb{#1}}\; }
\newcommand{\maprightspsb}[2]{\; \smash{\mathop{\; \longrightarrow \; }\limits\sp{#1}\limits\sb{#2}}\; }

\newcommand{\maprightinjsb}[1]{\; \smash{\mathop{\; \inj\; }\limits\sb{#1}}\; }

\newcommand{\mapdownsurj}{
\hbox{$\bigm\downarrow$}
\llap{\hbox{\raise 2pt\hbox{$\bigm\downarrow$}}}%
\vstrechmapdown
}

\newcommand{\mapupsurj}{
\hbox{$\bigm\uparrow$}
\llap{\hbox{\raise 2pt\hbox{$\bigm\uparrow$}}}%
\vstrechmapup
}

\newcommand{\inj}{\hookrightarrow}

\newcommand{\isom}{\xrightarrow{\sim}}

\newcommand{\set}[2]{\{\; {#1} \; \mid \; {#2} \;  \}}
\newcommand{\shortset}[2]{\{ {#1} \,|\, {#2}   \}}

\newcommand{\gen}[1]{\langle {#1}  \rangle}

\newcommand{\tensor}{\otimes}

\newcommand{\sprime}{\sp\prime}

\newcommand{\spprime}{\sp{\prime\prime}}

\newcommand{\sptimes}{\sp{\times}}
\newcommand{\sperp}{\sp{\perp}}

\newcommand{\dual}{\sp{\vee}}

\newcommand{\semidirectproduct}{\rtimes}

\newcommand{\inv}{\sp{-1}}

\newcommand{\GL}{\mathord{\mathrm{GL}}}

\newcommand{\PGU}{\mathord{\mathrm{PGU}}}
\newcommand{\PSU}{\mathord{\mathrm{PSU}}}
\newcommand{\PGL}{\mathord{\mathrm{PGL}}}
\newcommand{\PSL}{\mathord{\mathrm{PSL}}}
\newcommand{\OG}{\mathord{\mathrm{O}}}

\newcommand{\Ker}{\operatorname{\mathrm{Ker}}\nolimits}

\newcommand{\Image}{\operatorname{\mathrm{Im}}\nolimits}
\newcommand{\Aut}{\operatorname{\mathrm{Aut}}\nolimits}
\newcommand{\Gal}{\operatorname{\mathrm{Gal}}\nolimits}

\newcommand{\pr}{\mathord{\mathrm{pr}}}

\newcommand{\mystruth}[1]{\phantom{\hbox{\vrule height #1}}}

\newcommand{\mystruthd}[2]{\phantom{\hbox{\vrule  height #1 depth #2}}}

\newcommand{\intf}[1]{\langle #1\rangle}
\newcommand{\intfvoid}{\intf{\phantom{i}, \phantom{i}}}
\newcommand{\discg}{A}
\newcommand{\discf}{q}

\newcommand{\Leech}{\Lambda_{24}}
\newcommand{\Lamtt}{\Lambda_{22}}

\renewcommand{\PGL}{\mathord{\mathrm{PGL}}}
\renewcommand{\PSL}{\mathord{\mathrm{PSL}}}
\renewcommand{\PGU}{\mathord{\mathrm{PGU}}}
\newcommand{\PGamU}{\mathord{\mathrm{P}\Gamma\mathrm{U}}}
\newcommand{\PGamL}{\mathord{\mathrm{P}\Gamma\mathrm{L}}}
\renewcommand{\PSU}{\mathord{\mathrm{PSU}}}
\newcommand{\GU}{\mathord{\mathrm{GU}}}

\newcommand{\Ff}{\F_{4}}
\newcommand{\Ft}{\F_{2}}

\newcommand{\dotzero}{\mathord{\cdot} 0}

\newcommand{\dotttt}{\mathord{\cdot} 222}

\newcommand{\dottttAB}{\mathord{\cdot} 222_{AB}}

\newcommand{\circtttAB}{\mathord{\circ} 222_{AB}}

\newcommand{\thePSU}{\PSU(6,4)}
\newcommand{\thePGU}{\PGU(6,4)}

\newcommand{\thePGamU}{\PGamU(6,4)}

\newcommand{\mattranspose}{\lower 1.1pt \hbox{${}^{\mathrm T}$}\hskip 0pt}

\newcommand{\baromega}{\bar{\omega}}

\newcommand{\theD}{D}
\newcommand{\PPPX}{\PPP_X}
\newcommand{\NNNX}{\NNN_X}

\newcommand{\RAB}{R_{AB}}
\newcommand{\RABsperp}{R_{AB}\sperp}
\newcommand{\TTTAB}{\TTT_{AB}}

\newcommand{\nuPPP}{\nu_{\PPP}}
\newcommand{\nuTTT}{\nu_{\TTT}}
\newcommand{\rhoPPP}{\rho_{\PPP}}
\newcommand{\tauTTT}{\tau_{\TTT}}
\newcommand{\sigmaAB}{\sigma_{AB}}

\newcommand{\diffTLatAB}{\Lambda\sprime_{AB}}

\renewcommand{\discg}[1]{D_{#1}}
\renewcommand{\discf}[1]{q_{#1}}

\newcommand{\stdphi}{\phi_0}

\newcommand{\detsprime }{\mathord{\det}\sprime}

\newcommand{\dcolon}{\;\colon\,}
\newcommand{\zetaLam}{\zeta_{\Lambda}}
\newcommand{\BLam}{B_{\Lambda}}
\newcommand{\LLLX}{\LLL_X}

\newcommand{\midpt}{C}

\begin{document}
\title[On Edge's correspondence  associated with $\cdot 222$]%
{On Edge's correspondence associated with $\cdot 222$}
\author{Ichiro Shimada}
\address{Department of Mathematics, 
Graduate School of Science, 
Hiroshima University,
1-3-1 Kagamiyama, 
Higashi-Hiroshima, 
739-8526 JAPAN}
\email{ichiro-shimada@hiroshima-u.ac.jp}
\thanks{This work was supported by JSPS KAKENHI Grant Numbers JP16K13749, JP16H03926}

\begin{abstract}
We describe explicitly the correspondence of Edge
between the set of planes 
contained in the Fermat cubic $4$-fold in characteristic $2$,
and the set of lattice points $T$ of the Leech lattice $\Leech$
such that $OABT$ is a regular tetrahedron,
where $O$ is the origin of $\Leech$,  and $A$ and $B$ are fixed points of $\Leech$
such that $OAB$ is a regular triangle of edge length $2$.
Using this description,
we present Conway's isomorphism from $\thePSU$ to $\dotttt$ 
in terms of matrices.
\end{abstract}
\maketitle

\section{Introduction}
In Table~10.4~of Conway and Sloane ~\cite{CSbook}, 
it is shown 
that the subgroup $\dotttt$ of the orthogonal group $\dotzero=\OG(\Leech)$
of the Leech lattice $\Leech$ is isomorphic to the simple group $\thePSU$.
In~\cite{Ed70}, Edge constructed a permutation representation of $\dotttt$ 
on a certain set of lattice points of $\Leech$,
and suggested that this representation corresponds to
the natural permutation  representation of $\thePSU$ on the set $\PPPX$
of linear planes contained in  $X\tensor \overline{\F}_2$, 
where $X$ is 
the  Fermat cubic $4$-fold
$$
X\dcolon  x_1^3+\cdots+x_6^3=0
$$
defined over $\Ft$,
and $\overline{\F}_2$ is an algebraically closed field of characteristic $2$.
The purpose of this note is to clarify this correspondence 
by writing them explicitly.
Our idea is based on the investigation in~\cite{ShimadaPLMS}
of the lattice of   numerical equivalence classes of $n$-dimensional linear subspaces 
contained in  a $2n$-dimensional  
Hermite variety.
As an application, we write Conway's isomorphism
$\thePSU\cong \dotttt$ 
in terms of matrices.
It turns out that Conway's isomorphism
is closely related to the well-known isomorphism
between $\PSL(3, 4)$ and the Mathieu group $\M_{21}$
(see Proposition~\ref{prop:code}).
For simplicity, 
we put
\begin{equation}\label{eq:theD}
\theD:=9196830720=|\thePSU|=|\dotttt|.
\end{equation}
\par
First, we define a geometric object  $(\PPPX, \nuPPP)$ in $X$.
The hypersurface  $X\tensor \overline{\F}_2\subset \P^5$ contains exactly $891$ linear planes,
and they are all defined over $\Ff$ (see Segre~\cite{Seg65}). 
Let $\PPPX$ denote the set of all these planes.
For $\Pi, \Pi\sprime\in \PPPX$, we put
$$
\nuPPP(\Pi, \Pi\sprime):=\Pi\cdot\Pi\sprime=(1-(-2)^{\dim(\Pi\cap \Pi\sprime)+1})/3
$$
with  the understanding that $\dim(\emptyset)=-1$,
where $\Pi\cdot\Pi\sprime$ is  the intersection number  of the algebraic cycles $\Pi$ and $ \Pi\sprime$  of $X$.
The second equality in the above formula follows from the excess intersection formula~(see Chapter~6.3 of Fulton~\cite{FultonBook}).
For $a\in \overline{\F}_2$, 
let $a\mapsto \bar{a}:=a^2$ denote the Frobenius action
over $\Ft$,
and for a matrix $g$ with components in $\overline{\F}_2$,
let $\bar{g}$ denote the matrix obtained from $g$
by applying $a\mapsto \bar{a}$ to all the components of $g$.
We put
$$
 \GU(6, 4):=\set{g\in \GL(6, \overline{\F}_2)}{g\cdot \mattranspose{\bar{g}}=I_6},
 \quad
 \thePGU:=\GU(6, 4)/\Ff\sptimes.
 $$
Then $\thePGU$ is equal to the projective automorphism group of $X\tensor \overline{\F}_2\subset \P^5$,
and it acts on the set $X(\Ff)$ of $\Ff$-rational points of $X$.
Let $\thePGamU$ be
the subgroup of
the full symmetric group
of  $X(\Ff)$ 
generated by
the permutations induced by the  action of $ \thePGU$ and $\Gal(\Ff/\Ft)$.
We have
$$
|\thePGamU|=2\,|\thePGU|=6\theD.
$$
Since all $\Pi\in \PPPX$ are defined over $\Ff$, 
we have a natural homomorphism 
$$
\rhoPPP\dcolon \thePGamU \to \Aut(\PPPX, \nuPPP).
$$
\par
Next we define a lattice-theoretic object $(\TTTAB, \nuTTT)$
 in the Leech lattice $\Leech$.
Let $O$ denote the origin of $\Leech$.
For points $P, Q$ of $\Leech\tensor\R$,
we denote  by $|PQ|$ the length $\intf{PQ, PQ}^{1/2}$
of the vector  
$$
PQ:=Q-P, 
$$ 
where 
the inner-product $\intfvoid$ on $\Leech\tensor\R$
is induced from  the symmetric bilinear form of $\Leech$.
The vector $OP=P-O$ is sometimes simply denoted by $P$.
We put
$$
\AAA:=\set{[A, B]}{A, B\in \Leech, \; |OA|=|OB|=|AB|=2},
$$
whose cardinality $|\AAA|$ is equal to 
$4600\cdot 196560$.
Then $\dotzero$ acts on $\AAA$ transitively (see Chapter~10 of~\cite{CSbook}).
Let  $[A, B]$ be an arbitrary element of $\AAA$.
We put
$$
\TTTAB:=\set{T\in \Leech}{|OT|=|AT|=|BT|=2}.
$$
Let $\midpt$ be the mid-point of an edge of the regular triangle $OAB$.
(Note that $\midpt$ is not a lattice point of $\Leech$.)
For $T, T\sprime\in \TTTAB$, we put 
$$
\nuTTT(T, T\sprime):=\intf{\midpt T, \midpt T\sprime}=\intf{OT, OT\sprime}-1.
$$
Let $\diffTLatAB$ denote the sublattice of $\Leech$ generated by the vectors $TT\sprime$,
where $T$ and $T\sprime$ run through $\TTTAB$:
$$
\diffTLatAB:=\gen{\;TT\sprime \;\mid\;  T, T\sprime \in \TTTAB\; }.
$$
Then we obviously have a natural homomorphism
$$
\tauTTT\dcolon  \Aut (\TTTAB, \nuTTT) \to \OG(\diffTLatAB).
$$
\par
Edge~\cite{Ed70} observed that $|\PPPX|=|\TTTAB|=891$,
and that $\PPPX$ and $\TTTAB$  have many 
combinatorial properties in common.
\begin{definition}
 An \emph{Edge correspondence} is  
 a bijection $\phi\colon \PPPX\isom \TTTAB$ such that
 $\nuTTT(\phi(\Pi), \phi(\Pi\sprime))=\nuPPP(\Pi, \Pi\sprime)$ holds for any $\Pi, \Pi\sprime\in \PPPX$.
\end{definition}
Our first result is as follows:
\begin{theorem}\label{thm:main1}
An Edge correspondence exists.
\end{theorem}
\par
Two algebraic cycles $\Sigma, \Sigma\sprime$ of codimension $2$ on $X\tensor\overline{\F}_2$
are \emph{numerically equivalent} if $\Sigma\cdot \Sigma\spprime=\Sigma\sprime\cdot \Sigma\spprime$ holds
for any algebraic cycle $\Sigma\spprime$ of codimension $2$ on $X\tensor\overline{\F}_2$.
Let $\NNNX$ be the lattice of  numerical equivalence classes of algebraic cycles of codimension $2$ on $X\tensor\overline{\F}_2$.
The rank of  $\NNNX$ is equal to the 4th Betti number $b_4(X)=23$ of $X$;
that is, $X$ is supersingular (see Shioda and Katsura~\cite{ShiodaKatsura79}, Tate~\cite{Tate63}).
\begin{corollary}\label{cor:NNN}
The lattice $\NNNX$
is generated by the numerical equivalence classes of planes $\Pi\in \PPPX$, and 
hence is isomorphic to the lattice generated  by the vectors $\midpt T\in \Leech\tensor\Q$, where $T$ runs through $\TTTAB$.
\end{corollary}
We denote by $\RAB\subset \Leech$ the sublattice of rank $2$  containing $A, B$, and put 
%
%
$$
\dottttAB:=\shortset{g\in \dotzero}{A^g=A,\; B^g=B},
\quad
\circtttAB:=\shortset{g\in \dotzero}{\RAB^g=\RAB}.
$$
Then  $\circtttAB$ acts on the orthogonal complement $\RABsperp$ of $\RAB$ in $\Leech$.
The following lemma is proved in Section~\ref{subsec:lemma}.
See~\cite[Chapter 6]{CSbook} for the definition of laminated lattices.
\begin{lemma}\label{lem:Lam22}
The sublattice $\diffTLatAB=\gen{TT\sprime \mid T, T\sprime \in \TTTAB}$ of $\Leech$ is equal to  $\RABsperp$,
and is isomorphic to the laminated lattice $\Lamtt$ of rank $22$.
Moreover, the natural homomorphism
$$
\sigmaAB\dcolon  \circtttAB\to \OG(\RABsperp)=\OG(\diffTLatAB)\cong \OG(\Lamtt)
$$
is an isomorphism.
\end{lemma}
The following theorem gives a geometric  explanation of Conway's isomorphism  $\thePSU\cong \dotttt$
via an Edge correspondence.
\begin{theorem}\label{thm:main2}
Let $\phi\colon (\PPPX, \nuPPP)\isom (\TTTAB, \nuTTT)$ be an Edge correspondence.
Then the  composite homomorphism 
$$
\thePGamU  \maprightspsb{}{\rhoPPP} 
\Aut(\PPPX, \nuPPP) \maprightspsb{\sim}{\textrm{\rm by $\phi$}}   
\Aut(\TTTAB, \nuTTT) \maprightspsb{}{\tauTTT}
 \OG(\diffTLatAB) \maprightspsb{\sim}{\sigmaAB\inv}
 \circtttAB
$$
is injective, has the image of index $2$ in $\circtttAB$, and induces 
an isomorphism $\thePSU\cong \dottttAB$.
\end{theorem}
\begin{corollary}\label{cor:rhoPPP}
The homomorphism $\rhoPPP$ 
is an isomorphism.
\end{corollary}
\begin{corollary}\label{cor:Lam22}
$\OG(\Lamtt)/\{\pm I_{22}\}\cong \thePGamU$.
\end{corollary}
%
%
\par
We prove Theorem~\ref{thm:main1}
in Section~\ref{subsec:main1}
by presenting an example $\stdphi$ of the Edge correspondence.
Then 
we write  the homomorphism  $\thePGamU \to \circtttAB$  in Theorem~\ref{thm:main2} for $\phi=\stdphi$
in terms of matrices (Theorem~\ref{thm:main3}).
Using these matrices,
we prove Theorem~\ref{thm:main2} and Corollaries  in Section~\ref{subsec:main2}.
Since we state our results  explicitly,
most of them  can be checked by direct computation.
The computational data is available from the author's webpage~\cite{ShiCompEdge}.
For the computation,
we  used {\tt GAP}~\cite{GAP}.
\par
\medskip
Thanks are due to Professor Tetsuji Shioda and Professor Shigeru Mukai 
for discussions, and 
to Professor Ivan Cheltsov for inviting  me to write this paper.
\section{Constructing an Edge correspondence}\label{sec:bijection}
\subsection{Notation}
Let $S$ be a finite set.
For a subset $W\subset S$, let $v_W\colon S\to \Ft$ denote 
the function such that $v_W\inv(1)=W$.
By $W\mapsto v_W$,
we equip the power set $2^S$ of $S$ with a structure of the $\Ft$-vector space.
We thus regard 
a linear binary code on $S$
as a non-empty subset of $2^S$ closed under the symmetric difference.
\subsection{Fixing a basis of $\Leech$}
In order to be explicit,
we fix a basis of the Leech lattice $\Leech$.
Let $M:=\{1,2,\dots, 24\}$ be the set of positions of the MOG (Miracle Octad Generators, see Chapter~11 of~\cite{CSbook})
indexed by the following diagram.
$$
\begin{array}{|cc|cc|cc|}
\hline
\phantom{0}1 & \phantom{0}5 & \phantom{0}9  & {13} & {17} & {21} \\
\phantom{0}2 & \phantom{0}6 & {10}  & {14} & {18} & {22} \\
\hline
\phantom{0}3 & \phantom{0}7 & {11}  & {15} & {19} & {23} \\
\phantom{0}4 & \phantom{0}8 & {12}  & {16} & {20} & {24}\\
\hline
\end{array}
$$
Let $\Z^M$ be the $\Z$-module  of functions from $M$ to $\Z$, and we equip $\Z^M$ 
with the inner product
$$
(x_1 y_1+\cdots+x_{24} y_{24})/8.
$$
We define $\Leech$ to be the sublattice of $\Z^M$ generated by 
the row vectors  
of the matrix $\BLam$ in Figure 4.12 of~\cite{CSbook}, with the scalar multiplication  $1/\sqrt{8}$ removed.
Each point of $\Leech$
is written as a row vector with respect to the standard basis of $\Z^M$.
Hence each element of $\dotzero$
acts on $\Leech$  from the \emph{right},
and  is expressed  by a $24\times 24$ orthogonal matrix $g$ with components in $\Q$
whose  action  on $\Z^M\tensor \Q$
preserves $\Leech \subset \Z^M\tensor \Q$;
that is, we have 
$$
\dotzero=\set{g\in \GL_{24}(\Q)}{ g\cdot \mattranspose g=I_{24}, \;\;  \BLam \cdot g \cdot \BLam\inv \in  \GL_{24}(\Z)}. 
$$
%
%
%
\subsection{Two binary codes of length $21$}\label{subsec:21}
Let $\P^2$ be a projective plane over $\Ft$.
We denote by $S$ the set $\P^2(\Ff)$ of $\Ff$-rational points of $\P^2$.
Let $\HHH\subset 2^S$ be the linear code  generated by the codewords $\ell(\Ff)\subset S$ of weight $5$,
where $\ell$ runs through the set of $\Ff$-rational lines on $\P^2$,
and $\ell(\Ff)$ is the set of $\Ff$-rational points of $\ell$.
Then  $\HHH$ is of dimension $10$.
Let $\PGamL(3, 4)$ be the subgroup of the full symmetric group of $S$ generated by
the permutations induced by the actions of $\PGL(3, 4)$ and  $\Gal(\Ff/\Ft)$.
It is obvious that the group $\PGamL(3, 4)$  of order $120960$ acts on the linear code $\HHH$.
\par
Let $\CCC_{24}\subset 2^M$ be the extended binary Golay code
defined by MOG (see Chapter~11 of~\cite{CSbook}).
It is well-known that 
the automorphism group of $\CCC_{24}$ is the Mathieu group $\M_{24}$.
We put $M\sprime:=\{1,2,\dots, 21\}\subset M$, and let $\pr_{21}\colon 2^M\to 2^{M\sprime}$ be the natural projection.
We then put
$$
\CCC\sprime_{21}:=\set{w\in \CCC_{24}}{w(22)=w(23)=w(24)},
\quad
\CCC_{21}:=\pr_{21}(\CCC\sprime_{21}).
$$
\par
%
The following  well-known fact explains the  isomorphism between 
 $\PSL(3, 4)$ and the Mathieu group $\M_{21}=\shortset{\sigma\in \M_{24}}{n^\sigma=n\;\textrm{for}\; n=22,23,24}$.
\begin{proposition}\label{prop:code}
{\rm(1)} The binary codes $\HHH$ and $\CCC_{21}$ are isomorphic.
The weight distribution of these codes is
$$
0^1\, 5^{21} \, 8^{210} \, 9^{280} \,12^{280} \, 13^{210}\, 16^{21} \, 21^1.
$$
{\rm(2)} The automorphism group of $\HHH$ is equal to $\PGamL(3, 4)$.
In particular, there exist exactly $|\PGamL(3, 4)|=120960$ isomorphisms between $\HHH$ and $\CCC_{21}$.
\end{proposition}
\subsection{Proof of Theorem~\ref{thm:main1}}\label{subsec:main1}
The condition $\nuPPP(\Pi, \Pi\sprime)=\nuTTT(\phi(\Pi), \phi(\Pi\sprime))$
required for an Edge correspondence $\phi\colon \PPPX\isom \TTTAB$
is equivalent to the condition that the third  $\mathord{\Longleftrightarrow}$ in each row of the following table should hold:
$$
\begin{array}{ccccccc}
\dim (\Pi\cap \Pi\sprime) && |\Pi(\Ff)\cap\Pi\sprime(\Ff)| && \Pi\cdot \Pi\sprime && \intf{\phi(\Pi), \phi(\Pi\sprime)}\\
\hline
-1 &\Longleftrightarrow&0&\Longleftrightarrow&0 &\Longleftrightarrow & 1\phantom{.}\\
0&\Longleftrightarrow&1&\Longleftrightarrow&1 &\Longleftrightarrow & 2\phantom{.}\\
1 &\Longleftrightarrow&5&\Longleftrightarrow&-1 &\Longleftrightarrow & 0\phantom{.}\\
2 &\Longleftrightarrow&21&\Longleftrightarrow&3 &\Longleftrightarrow & 4.
\end{array}
$$
\par
We construct  an example $\stdphi $
of the Edge correspondence 
explicitly.
Let $\omega\in \F_4$ be a root of $x^2+x+1=0$, and 
let $\Pi_0, \Pi_{\infty} \in \PPPX$ be the planes specified by three points on them as follows:
\begin{eqnarray*}
\phantom{aaa} \Pi_0\,\,&=&\gen{\;\; (1:0:0:\omega:0:0), \;\;(0:1:0:0:\omega:0), \;\; (0:0:1:0:0:\omega)\;\;}, \label{eq:fixPP1}\\
\phantom{aaa} \Pi_{\infty}&=&\gen{\;\;(1:0:0:\baromega:0:0), \;\; (0:1:0:0:\baromega:0), \;\; (0:0:1:0:0:\baromega)\;\;}. \label{eq:fixPP2}
\end{eqnarray*}
Note that $\nuPPP(\Pi_0, \Pi_{\infty})=0$.
Then there exist exactly $21$ planes $\Pi_s\in \PPPX$ such that
$$
\dim (\Pi_0\cap \Pi_s)=1,\quad 
\dim (\Pi_{\infty}\cap \Pi_s)=-1.
$$
Let $S=\{\Pi_1, \dots, \Pi_{21}\}$ be the set of these  $21$ planes.
The mapping $\Pi_s\mapsto \Pi_s\cap \Pi_0$
establishes a bijection from $S$ to the set 
$\Pi_0\dual(\Ff)$
of  $\F_4$-rational lines on $\Pi_0$,
where $\Pi_0\dual$ is the dual plane of $\Pi_0$.
Via this bijection and $\Pi_0\dual\cong\P^2$, 
the notation $S=\{\Pi_1, \dots, \Pi_{21}\}$   is compatible with the notation $S=\P^2(\Ff)$ in Section~\ref{subsec:21}.
For each $\Ff$-rational point $P$ of $\Pi_0$,
the codeword
$$
\ell_P:=\set{\Pi_s\in S}{P\in \Pi_s}\in 2^S
$$
is of weight $5$.
Each codeword $\ell_P$ is an $\Ff$-rational  line of $S=\Pi_0\dual(\Ff)$, and 
these $21$ codewords $\ell_P$
generate a linear code $\HHH_X\subset 2^S$ of dimension $10$.
\par
Recall that each vector of $\Leech$ is written as a row vector with respect to the standard basis of $\Z^{M}$.
Since $\dotzero$ acts on $\AAA$ transitively (see Chapter~10 of~\cite{CSbook}), 
we can assume without loss of generalities that 
the fixed lattice points  $A, B\in \Leech$ are 
\begin{equation}\label{eq:fixAB}
A=(0^{21}, 4, 0, -4), \quad B=(0^{21}, 0, 4, -4),
\end{equation}
where $0^{21}$ is the constant map to $0$ from $M\sprime=\{1, \dots, 21\}\subset M$,
and the last three coordinates indicates the values at $22,23,24\in M$.
In the following, we will use this choice of $A$ and $B$.
It is easy to calculate  the set $\TTTAB$.
It turns out that  $\TTTAB$ consists of the following $891$ vectors:
$$
\begin{array}{lll}
\textrm{one  element of type} & (0^{21}, 4,4,0) & \textrm{(type 0),} \\
\textrm{$42$ elements of  type} & (0^{20} (\pm 4)^1, 0,0,-4) & \textrm{(type 1),}\\
\textrm{$336$  elements of  type} & (0^{16} (\pm 2)^5, 2, 2, -2) & \textrm{(type 2),} \\
\textrm{$512$  elements of  type} & ((\pm 1)^{21}, 1,1,-3)  & \textrm{(type 3)},
\end{array}
$$
where, for example, $0^{16} (\pm 2)^5$ indicates a map from $M\sprime$ to $\{0, \pm 2\}$
whose fiber over $0$ is of cardinality $16$.
These vectors are described more precisely as follows.
\begin{itemize}
\item For a vector of type 1, the position of $(\pm 4)^1$ is arbitrary in $M\sprime$ and the sign is also arbitrary.
\item For a vector of type 2, the positions of $(\pm 2)^5$ form a subset $W$ of  $M\sprime$ such that $W\cup\{22,23,24\}$
is an octad of the Golay code $\CCC_{24}$,  and the number of minus sign is odd.
\item For a vector of type 3, the positions of minus sign in $(\pm 1)^{21}$ form a subset $W$ of  $M\sprime$
that is a code word of $\CCC_{24}$ disjoint from $\{22,23,24\}$.
\end{itemize}
Let $T_0\in \TTTAB$ and $T_{\infty}\in \TTTAB$ be the lattice points
\begin{equation*}\label{eq:fixTT}
T_0=(0^{21}, 4, 4, 0), \quad T_{\infty}=(1^{21}, 1,1,-3).
\end{equation*}
Note that $\nuTTT(T_0, T_{\infty})=0$.
The set 
 of all $T\in \TTTAB$ satisfying   
$$
\intf{T_0, T}=0, \quad \intf{T_{\infty}, T}=1
$$
is equal to the set  of the points 
$$
T_{i}=(0^{20} (-4), 0,0, -4),
$$
where $-4$ is located at the $i$th position for $i=1, \dots, 21$.
By $T_i\mapsto i$,
we identify this set  with $M\sprime=\{1, \dots, 21\}$.
A point $T\in \TTTAB$  satisfies $\intf{T, T_0}=2$ if and only if $T$ is of type 2 above,
and when this is the case, the codeword
$$
F_T:=\set{T_i\in M\sprime}{\intf{T, T_i}\in \{0, 2\}}\;\; \in \;\; 2^{M\sprime}.
$$
is equal to the set of positions
$j$ such that the $j$th component of $T$ is $\pm 2$.
Hence, by the definition of  $\CCC_{21}$, 
 the binary code generated by these  $F_T$  is equal to $\CCC_{21}$.
\par
By Proposition~\ref{prop:code}, 
there exist exactly $120960$  bijections $\varphi\colon  S\isom M\sprime$
that induce $\HHH_X\cong \CCC_{21}$.
We choose the following isomorphism $\varphi_0$.
As noted above,
we identify $S=\{\Pi_1, \dots, \Pi_{21}\}$  with $\Pi_0\dual (\F_4)$.
We use $(x_4:x_5:x_6)$ as the homogeneous coordinates of $\Pi_0$,
and let $[\xi_4:\xi_5:\xi_6]$ be the homogeneous coordinates of $\Pi_0\dual$ dual to $(x_4:x_5:x_6)$.
Then the map $\varphi_0\colon  S=\Pi_0\dual (\F_4)\isom M\sprime$  given by
the diagram below 
induces   $\HHH_X\cong \CCC_{21}$.
$$
\renewcommand{\arraystretch}{1.2}
\begin{array}{|cc|cc|cc|}
\hline
{[}0:1:0{]} & {[}1:0:0{]} & {[}1:\baromega:0{]}  & {[}1:\omega:0{]} & {[}1:1:0{]} & {[}0:0:1{]} \\
{[}0:1:1{]} & {[}1:0:1{]} & {[}1:\baromega:\omega{]}  & {[}1:\omega:\baromega{]} & {[}1:1:1{]}&\\
\hline
{[}0:1:\omega{]} & {[}1:0:\omega{]} &  {[}1:\baromega:\baromega{]}  & {[}1:\omega:1{]} & {[}1:1:\omega{]} & \\
{[}0:1:\baromega{]}  & {[}1:0:\baromega{]}& {[}1:\baromega:1{]} & {[}1:\omega:\omega{]}  &  {[}1:1:\baromega{]}&\\
\hline
\end{array}
$$
\par
We define $\stdphi$ on the subset $\{\Pi_0, \Pi_{\infty}\}\cup S$ of $\PPPX$ 
by  $\stdphi(\Pi_0)=T_0$, $\phi(\Pi_{\infty})=T_{\infty}$,
and $\stdphi |S=\varphi_0$.
Then it is a matter of simple calculation to show that, 
for each plane $\Pi\in \PPPX\setminus (\{\Pi_0, \Pi_{\infty}\}\cup S)$,
there exists a unique  point $\stdphi(\Pi)\in \TTTAB$ that satisfies
$\nuPPP(\Pi, \Pi\sprime)=\nuTTT(\stdphi(\Pi), \stdphi(\Pi\sprime))$
for all $\Pi\sprime \in \{\Pi_0, \Pi_{\infty}\}\cup S$,
 and that the map 
$\stdphi\colon \PPPX\to \TTTAB$ thus constructed  
is an Edge correspondence.
\par
We describe the vector representation $(x_1, \dots, x_{24})\in \Z^M$ of 
$\stdphi(\Pi)\in \TTTAB$
for each $\Pi\in \PPP_X$ in the table below,
where $\delta_{0}=\dim(\Pi\cap \Pi_{0} )$, $\delta_{\infty}=\dim(\Pi\cap \Pi_{\infty} )$,
and, for example, the entry $280: (-1)^{12} 1^9$ means that
there exist exactly $280$ planes $\Pi\in \PPPX$ with $\delta_0=\delta_{\infty}=-1$,  
and that they are mapped to the lattice points  of the form $((-1)^{12} 1^9, 1,1,-3)$.
$$
\setlength{\arraycolsep}{8pt}
\begin{array}{c|cccc}
\delta_{\infty} \backslash\delta_0 & -1 & 0 & 1 & 2 \\
\hline
-1 & 280: (-1)^{12} 1^9 & 210: 0^{16} (-2)^3 2^2 & 21: 0^{20} (-4) & 1: 0^{21} \mystruth{11pt}\\
0 & 210: (-1)^{8} 1^{13} & 105:  0^{16} (-2) 2^4 & 21: 0^{20} 4& \\
1 & 21: (-1)^{16} 1^5 & 21: 0^{16} (-2)^5 &  & \\
2 & 1: 1^{21} &          & & \\
\hline
(x_{22}, x_{23}, x_{24}) & (1,1,-3) & (2,2,-2) & (0,0,-4)&  (4,4,0) \mystruth{11pt}
\end{array}
$$
\begin{remark}\label{rem:adhoc1}
The above construction of $\stdphi$ involves  seemingly \emph{ad hoc} choices of $\Pi_0, \Pi_{\infty}$,
$T_0, T_{\infty}$, 
and the code isomorphism $\varphi_0$.
In Remark~\ref{rem:adhoc2}, we show that these choices can be made arbitrarily
for the construction of an Edge correspondence.
\end{remark}
\section{The  representation of $\thePGamU$ in the Leech lattice}
\subsection{Proof of Lemma~\ref{lem:Lam22}}\label{subsec:lemma}
We recall the notion of \emph{discriminant forms} due to Nikulin~\cite{theNikulin}.
Let $L$ be an even lattice; that is,  $\intf{v, v}\in 2\Z$ holds for all $v\in L$.
Then $L$ is naturally embedded  into the dual lattice $L\dual:=\shortset{x\in L\tensor\Q}{\intf{x, v}\in \Z\;\textrm{for all}\;v\in L}$
as a submodule of finite index.
We call $\discg{L}:=L\dual/L$ the \emph{discriminant group} of $L$.
By $\discf{L}(x \bmod L):=\intf{x, x}\bmod 2\Z$, we obtain  a  finite quadratic form
$$
\discf{L}\dcolon  \discg{L}\to \Q/2\Z,
$$
which is called the \emph{discriminant form} of $L$.
In the following,
we denote by $\OG(\discf{L})$ the automorphism group
of the finite quadratic  form $\discf{L}\colon \discg{L}\to \Q/2\Z$,
which we let act on $\discf{L}$ from the right,
and by
 $$
 \eta\colon \OG(L)\to \OG(\discf{L})
 $$
the natural homomorphism.
\par
It is obvious that
$\diffTLatAB=\gen{TT\sprime \mid T, T\sprime \in \TTTAB}$ is contained in $\RABsperp$.
We can calculate that the orders of 
the discriminant groups of $\diffTLatAB$ and of $\RABsperp$ are both $12$.
Hence  $\diffTLatAB=\RABsperp$ holds.
By the choice of  $A$ and $B$   in~\eqref{eq:fixAB},
the sublattice  $\diffTLatAB=\RABsperp$ is the section of $\Leech$ by the linear subspace 
of $\Leech\tensor\R$ defined by
$x_{22}=x_{23}=x_{24}$.
By the definition of $\Lamtt$ in Figure 6.2 of~\cite{CSbook}, 
we obtain $\diffTLatAB=\Lamtt$.
(This fact has been already proved in~\cite{ShimadaPLMS}.)
\par
The orthogonal group $\OG(\RAB)$ is isomorphic to the dihedral group  of order $12$.
Indeed, there exist exactly $6$ vectors of square norm $4$ in $\RAB$,
and they form a regular hexagon.
By direct calculation, we see the following:
\begin{equation}\label{eq:ORGisom}
\textrm{The natural homomorphism $\eta\colon \OG(\RAB)\to \OG(\discf{\RAB})$
is an isomorphism.}
\end{equation}
By~Nikulin~\cite{theNikulin}, the even unimodular overlattice $\Leech$ of the orthogonal direct sum 
$\RABsperp\oplus \RAB$ defines an anti-isometry
$$
\zetaLam\dcolon  \discf{\RABsperp} \isom -\discf{\RAB},
$$
and that $\circtttAB$ is identified with
$$
\set{(g_1, g_2)\in \OG(\RABsperp)\times \OG(\RAB)}{\eta(g_1)\circ\zetaLam=\zetaLam\circ \eta(g_2)}.
$$
By~\eqref{eq:ORGisom}, 
we see that the first projection $\sigmaAB\colon \circtttAB\to \OG(\RABsperp)$ is an isomorphism.
Thus Lemma~\ref{lem:Lam22} is proved.
\begin{remark}\label{rem:invalgo}
The proof above indicates how to calculate $\sigmaAB\inv\colon \OG(\RABsperp) \to \circtttAB$.
Let $g_1\in \OG(\RABsperp)$ be given.
We calculate $\eta({g}_1)\in \OG(\discf{\RABsperp})$,
and transplant $\eta({g}_1)$ to $\eta({g}_2)\in \OG(\discf{\RAB})$  via $\zetaLam$.
Then there exists a unique isometry $g_2\in \OG(\RAB)$ that induces $\eta({g}_2)$.
The action of $(g_1, g_2)$  on $\RABsperp\oplus \RAB$
preserves the overlattice $\Leech$,
and hence induces $g\in \OG(\Leech)$,
which belongs to  $\circtttAB$ by definition.
\end{remark}
\subsection{Computation of  $\thePGamU\to \circtttAB$}
For an Edge correspondence $\phi$,
we denote by
$$
\Psi_{\phi}\dcolon  \thePGamU\to \circtttAB
$$
the composite homomorphism in Theorem~\ref{thm:main2}.
We present $\Psi_0:=\Psi_{\stdphi}$  in terms of matrices
for the Edge correspondence  $\stdphi$  constructed in Section~\ref{subsec:main1}.
\par
We let $\PGU(6, 4)$ act on $\P^5$ from the right.
Taylor~\cite{Tay87} gave a generating set $\{\alpha\sprime, \beta\sprime\}$ 
of the group
 $$
 \GU(6, 4)\sprime:=\set{g\in \GL(6, 4)}{g \cdot J_6\cdot \mattranspose{\bar{g}}=J_6},
 $$
where $J_6$ denotes the Hermitian form
$x_1 \bar{x}_6+x_2 \bar{x}_5+\cdots+x_6 \bar{x}_1$ on $\P^5$ defined over $\Ft$.
From this result, we see that 
 $\GU(6, 4)$
 is generated by the following two elements of order $3$ and $10$, respectively.
 $$
 \alpha:=
 \left[
 \begin{array}{cccccc}
\omega&  0&  0&  0&  0&  0 \\
 0&  1&  0&  0&  0&  0 \\ 
 0&  0&  1&  0&  0&  0 \\
 0&  0&  0&  1&  0&  0 \\ 
 0&  0&  0&  0&  1&  0 \\
 0&  0&  0&  0&  0&  \omega 
  \end{array}
  \right],
  \quad
  \beta:=
    \left[
  \begin{array}{cccccc}
\baromega& 0& 0& 1& 0& 1 \\
 1& 0& 1& \omega& 0& 0\\
 0& 1& 0& 0& 0& 0 \\
 0& 0& 0& 0& 1& 0 \\
 0& 0& \baromega& 1& 0& 1 \\
 1& 0& 1& 0& 0& \omega 
  \end{array}\right].
 $$
 Let $\gamma\in \thePGamU$
denote  the Frobenius action of $\Gal(\Ff/\Ft)$.
Then $\thePGamU$ is generated by $\alpha, \beta, \gamma$.
By direct calculation, we obtain the following:
\begin{theorem}\label{thm:main3}
Let $\tilde{\alpha}$, $\tilde{\beta}$ 
and $\tilde{\gamma}$ be the 
three matrices in Figure~\ref{figure:threemats} . 
Then the homomorphism $\Psi_0$ is given by 
$\alpha\mapsto \tilde{\alpha},
\beta\mapsto \tilde{\beta},
\gamma\mapsto \tilde{\gamma}$.
\end{theorem}
{\renewcommand{\baselinestretch}{1}
\begin{figure}[H]
{\tiny
$$
\setlength{\arraycolsep}{2pt}
\hbox{\LARGE  $\frac{1}{4}$\;}
 \input figmatA
$$
\vfill
$$
\setlength{\arraycolsep}{2pt}
\hbox{\LARGE  $\frac{1}{8}$\;}
 \input figmatB
$$
\vfill
$$
\setlength{\arraycolsep}{2pt}
\hbox{\LARGE  $\frac{1}{8}$\;}
 \input figmatC
$$
}
\caption{$\tilde{\alpha}$, $\tilde{\beta}$, $\tilde{\gamma}$}\label{figure:threemats}
\end{figure}
}
\subsection{Proof of Corollary~\ref{cor:NNN}}
%
Let $\LLLX$ be the sublattice of $\NNNX$ generated by
the numerical equivalence classes  of the planes $\Pi\in \PPPX$.
Theorem~\ref{thm:main1}
implies  that $\LLLX$ is isomorphic to the lattice generated by the vectors $\midpt T\in \Leech\tensor\Q$,
where $\midpt$ is the mid-point of an edge of $OAB$, and $T$ runs through $\TTTAB$.
We show that $\LLLX=\NNNX$.
A direct calculation shows 
 that the discriminant group $\discg {\LLLX}$ of $\LLLX$  is isomorphic to $\F_2^2$, and that 
 the discriminant form $\discf {\LLLX}\colon \discg {\LLLX}\to \Q/\Z$ is given by
 $$
 \left[
 \begin{array}{cc}
 0 & 1/2 \\ 1/2 &0
 \end{array}\right].
 $$
Note that $\thePGU$ acts on $\LLLX\inj \NNNX$ equivariantly.
If $\LLLX\ne \NNNX$,
then $\NNNX/\LLLX$ would be  a non-zero $\thePGU$-invariant isotropic subspace
of $\discf {\LLLX}$.
However,
by looking at the action of $\alpha\in \thePGU$
on $\discg{\LLLX}$,
we conclude that there exists no such isotropic subspace.
\subsection{Proof of Theorem~\ref{thm:main2} and Corollaries~\ref{cor:rhoPPP} and~\ref{cor:Lam22}}\label{subsec:main2}
By Lemma~\ref{lem:Lam22}, the homomorphism $\tauTTT$ is regarded as a homomorphism to $ \OG(\RABsperp)$. 
\begin{lemma}\label{lem:doesnotcontain}
The natural homomorphism $\tauTTT\colon \Aut(\TTTAB, \nuTTT)\to \OG(\RABsperp)$
is injective, and the image does not contain $-I_{22}$.
\end{lemma}
\begin{proof}
Let $\SSS_0$ be the set of pairs $[T, T\sprime]$
with $T, T\sprime\in \TTTAB$ and $\intf{T, T\sprime}=0$.
We have $|\SSS_0|=891\cdot 42$.
Let $\pr_{R\sperp}\colon \Leech\tensor\Q\to \RABsperp\tensor\Q$
be the orthogonal projection.
We consider  the map 
$$
s_0\dcolon   [T, T\sprime]\mapsto \pr_{R\sperp}(T T\sprime)
$$
from $\SSS_0$ to $\RABsperp$.
By direct calculation, 
we see that  $s_0$ is injective, and its image spans $\RABsperp\tensor\Q$.
Hence an automorphism of $(\TTTAB, \nuTTT)$ belonging to  the kernel of $\tauTTT$
acts on $\SSS_0$ trivially.
Therefore $\tauTTT$ is injective.
If $-I_{22}$ were in the image of $\tauTTT$,
then there would exist an automorphism of $(\TTTAB, \nuTTT)$
that acts on $\SSS_0$ as $[T, T\sprime]\mapsto [T\sprime, T]$,
which is absurd.
\end{proof}
\begin{proof}[Proof of Theorem~\ref{thm:main2} and Corollaries~\ref{cor:rhoPPP} and~\ref{cor:Lam22}]
First we prove the assertion of Theorem~\ref{thm:main2} for $\Psi_0$;
that is, we are in the case where $\phi=\stdphi$.
Since $\lambda^6=1$ for all $\lambda\in \Ff\sptimes$,
it follows that  $\det\colon \GL(6,4)\to \Ff\sptimes$
factors as $\GL(6,4)\to \PGL(6, 4)\to \Ff\sptimes$.
Hence we have a homomorphism
$$
\detsprime \dcolon  \thePGamU=\thePGU\semidirectproduct \Gal(\Ff/\Ft)   \to \Ff\sptimes \semidirectproduct \Gal(\Ff/\Ft),
$$
whose kernel is  $\thePSU$.
We consider the composite homomorphism
$$
\psi_0\dcolon  \thePGamU \maprightsb{\Psi_{0}}  \circtttAB  \maprightsp{}  \OG(\RAB),
$$
where $\circtttAB\to \OG(\RAB)$ is the natural homomorphism.
Note that the kernel of $\circtttAB\to \OG(\RAB)$  is $\dottttAB$.
We denote elements of  $\OG(\RAB)$, which acts on $\RAB$ from the right, by the matrices
with respect to the basis $OA$ and $OB$.
Using the three matrices $\Psi_0(\alpha)=\tilde{\alpha}, \Psi_0(\beta)=\tilde{\beta}, \Psi_0(\gamma)= \tilde{\gamma}$
and the algorithm in Remark~\ref{rem:invalgo},
we obtain the following:
$$
\begin{array}{c|ccc}
&\alpha & \beta & \gamma \\
\hline 
\detsprime &\baromega & 1 & \textrm{Frobenius} \mystruthd{10pt}{8pt}\\
\psi_0 & \left[\begin{array}{cc} 0&-1\\1&-1\end{array}\right] & \left[\begin{array}{cc} 1&0\\0&1\end{array}\right]  & \left[\begin{array}{cc} 1&-1\\0&-1\end{array}\right] 
\end{array}
$$
Since $\detsprime(\beta)=1$ and $\Psi_{0}(\beta)=\tilde\beta\ne 1$,
we see that $\thePSU\not\subset \Ker \Psi_{0}$.
Since $\thePSU$ is simple, we have $\thePSU\cap \Ker \Psi_{0}=1$.
Hence $\detsprime$ embeds $\Ker \Psi_{0}$ into $\Ff\sptimes \semidirectproduct \Gal(\Ff/\Ft)$.
The table above also shows that the mapping 
$$
\baromega\mapsto \psi_0(\alpha),\;\;\;
1\mapsto \psi_0(\beta),\;\;\;
\textrm{Frobenius}\mapsto \psi_0(\gamma)
$$
induces
an injective homomorphism
$$
i\dcolon \Ff\sptimes \semidirectproduct \Gal(\Ff/\Ft) \inj  \OG(\RAB)
$$
 such that 
$\psi_0$ factors as 
$$
\thePGamU \maprightsb{\detsprime } \Ff\sptimes \semidirectproduct \Gal(\Ff/\Ft) \maprightinjsb{i} \OG(\RAB).
$$
Therefore  $\Ker \Psi_0$ is trivial.
Plesken and Pohst~\cite{PleskenPohst93} showed that $| \OG(\Lamtt)|=12\theD$,
where $\theD$ is given in~\eqref{eq:theD}.
By Lemma~\ref{lem:Lam22}, we have
$|\circtttAB|=12\theD$.
By Lemma~\ref{lem:doesnotcontain},
we see that $\Image \Psi_0$ does not contain $-1$ and hence $|\Image \Psi_0|\le 6\theD=|\thePGamU|$.
Therefore  $\rhoPPP$ is an isomorphism, $|\Image \Psi_0|= 6\theD$, and $\circtttAB$ is generated by $\Image \Psi_0$ and $-1$.
Moreover
 $\Psi_0$ induces an isomorphism from  $\thePSU$ to $\dottttAB$.
Thus  the assertion of Theorem~\ref{thm:main2} for $\Psi_0$
and Corollaries~\ref{cor:rhoPPP} and~\ref{cor:Lam22} are proved.
\par
Let $\phi$ be an arbitrary Edge correspondence.
Since $\rhoPPP$ is an isomorphism,
the homomorphism $\Psi_{\phi}$ differs from $\Psi_0$ only by an inner automorphism of $\thePGamU$.
Hence Theorem~\ref{thm:main2} holds  for $\phi$.
\end{proof}	
\begin{remark}\label{rem:adhoc2}
We prove the assertion made in Remark~\ref{rem:adhoc1}.
We put 
$$
\Xi:=\set{[\Pi, \Pi\sprime]}{\Pi, \Pi\sprime\in \PPPX, \;\dim (\Pi\cap \Pi\sprime)=-1}.
$$
We have $|\Xi|=891\cdot 512$.
For each $[\Pi, \Pi\sprime]\in \Xi$, we put
$$
S_{[\Pi, \Pi\sprime]}:=\set{\Pi\spprime\in \PPPX}{\dim (\Pi\cap \Pi\spprime)=1, \;\dim (\Pi\sprime\cap \Pi\spprime)=-1}
$$
Let $\HHH_{[\Pi, \Pi\sprime]}\subset 2^{S_{[\Pi, \Pi\sprime]}}$ be the linear code
generated by the codewords
$$
\ell_P:=\set{\Pi\spprime\in S_{[\Pi, \Pi\sprime]}}{P\in \Pi\spprime}, 
$$
where $P$ runs through $\Pi(\Ff)$.
We can easily prove that $\thePGamU$ acts on $\Xi$ transitively.
Let $G_{[\Pi, \Pi\sprime]}$ be  the stabilizer subgroup  of $[\Pi, \Pi\sprime]\in \Xi$ in $\thePGamU$,
whose order is $6\theD/891/512=120960=|\PGamL(3, 4)|$.
Then $G_{[\Pi, \Pi\sprime]}$ acts on $S_{[\Pi, \Pi\sprime]}$ and on  $\HHH_{[\Pi, \Pi\sprime]}$.
We can also prove by direct calculation that the action of $G_{[\Pi, \Pi\sprime]}$ on $S_{[\Pi, \Pi\sprime]}$ is faithful
(see also~\cite{ShimadaPLMS}).
Hence $G_{[\Pi, \Pi\sprime]}\to \Aut(\HHH_{[\Pi, \Pi\sprime]})\cong \PGamL(3, 4)$ is an isomorphism.
Since $\thePGamU\cong \Aut(\PPPX, \nuPPP)\cong \Aut(\TTTAB, \nuTTT)$,
it follows  that $\Aut(\TTTAB, \nuTTT)$ acts on the set 
$$
\phi_0(\Xi)=\set{[T, T\sprime]}{T, T\sprime\in \TTTAB, \;\intf{T, T\sprime}=1}
$$
 transitively.
Therefore,
in the construction of the Edge correspondence $\phi$,
the choices of $\Pi_0, \Pi_{\infty}$,
$T_0, T_{\infty}$, 
and the code isomorphism $\varphi_0$ can be made arbitrarily.
\end{remark}
%

%
%
%
\end{document}